\documentclass[12pt,twoside]{amsart}
\usepackage{amsmath}
\usepackage{amssymb}
\usepackage{amscd}
\usepackage{xypic}
\newmuskip\pFqmuskip

\usepackage{amsthm}
\usepackage{graphicx}
\usepackage{enumerate}
\usepackage{graphicx}
\usepackage{tikz-cd}
\usepackage[new]{old-arrows}
\usepackage[cmtip,all]{xy}

\usepackage{amsmath, amsthm, amssymb}

\xyoption{all}
\title[Hodge-Arakelov inequalities for family of surfaces fibered by curves]{Hodge-Arakelov inequalities for family of surfaces fibered by curves}

\author{Mohammad Reza Rahmati}
\thanks{}
\address{ Universidad De La Salle Bajío, Campestre - Le\'on, Guanajuato, Mexico.
\hfill\break 
\hfill\break \\
\hfill\break }
\email{mrahmati@cimat.mx}

\newcommand{\comments}[1]{}

\newtheorem{theorem}{Theorem}[section]
\newtheorem{proposition}[theorem]{Proposition}

\newtheorem{lemma}[theorem]{Lemma}

\newtheorem{remark}[theorem]{Remark}
\newtheorem{example}[theorem]{Example}


\begin{document}

\begin{abstract}
The Hodge numerical invariants of a variation of Hodge structure over a smooth quas--projective variety are a measure of complexity for the global twisting of the limit mixed Hodge structure when it degenerates. These invariants appear in inequalities which they may have correction terms, called Arakelov inequalities. One may investigate the correction term to make them into equalities, also called Arakelov equalities. We investigate numerical Arakelov type (in)equilities for a family of surfaces fibered by curves. Our method uses Arakelov identities in a weight 1 and also in a weight 2 variations of Hodge structure (cf. \cite{GGK}), in a commutative triangle of fibrations. We have  proposed to relate the degrees of the Hodge bundles in the two families. We also compare the Fujita decomposition of Hodge bundles in these fibrations. We examine various identities and relations between Hodge numbers and degrees of the Hodge bundles in different levels. 
\end{abstract}

\maketitle

\section{Introduction}

Of important invariants of a variation of Hodge structure (VHS) that measure the complexity of their twist are the chern classes of their Hodge bundles. A question in this line is how the degeration of VHS affect these chern classes. The numerical data of a geometric variation of Hodge structure also  provide tools to study the iso-triviality of the family. They also give bounds to determine global invariants associated to the singularities of the fibrations. For example one can answer simple questions on fibrations in low dimensions that, if they must have singular fibers and even some bounds for the number of singular fibers. The Arakelov identities involve the degrees of the Hodge bundles in VHS given by geometric fibration. In fact the vanishing of the degrees of all the Hodge bundle give a criteria to say whether the fibration is iso-trivial. These equalities have been systematically studied in low dimensions 1, 2, 3 and 4, \cite{GGK, VZ}. 
    
Assume we have a fibration $f:X \to S$ of smooth projective varieties with $S$ being of dimension 1. For each $k$ it defines a variation of Hodge structure $(\mathcal{H}^{(k)}=R^kf_*\mathbb{C},F^p, \nabla)$ of weight $k$ over $S$ which is the Deligne extension of that over the smooth locus of the fibration. Then the graded sheaves $(\mathcal{H}^{p,q}=F^p/F^{p-1}, \theta_q:=Gr_F^p \nabla)$ define a Higgs bundle. Following \cite{GGK}, in a weight one VHS $\mathcal{H}_{/S}$ obtained from the fibration by curves we have the following exact sequence of sheaves

\begin{equation}
0 \to \mathcal{H}_{0,e}^{1,0} \to \mathcal{H}_e^{1,0} \stackrel{\theta}{\longrightarrow} \check{\mathcal{H}}_e^{1,0} \otimes \Omega_S^1(\log E) \to \mathcal{B} \to 0
\end{equation}

\noindent
where $\mathcal{H}_e$ is the extended Hodge bundle and $\mathcal{H}_{0,e}^{1,0}:=\ker({\mathcal{H}}_e^{1,0} \to \check{\mathcal{H}}_e^{1,0} \otimes \Omega_S^1)$. We denote the rank of its generic fiber by $h_0^{1,0}$. The map $\theta$ is induced by the Gauss-Manin connection and is called the Kodaira-Spencer map. A calculation of the degree $\delta=\deg(\mathcal{H}^{1,0})$, [see  \cite{GGK}] gives the following general formula

\begin{equation}
\delta=\frac{1}{2}(h^{1,0}-h_0^{1,0})(2g(S)-2)+\delta_0-\frac{1}{2}\sum_{s \in S} \nu_s(\bar{\theta})
\end{equation}

\noindent
where $\delta_0=\deg(\mathcal{H}_{0,e}^{1,0})$, and $\nu_s(\bar{\theta}):=\dim \text{coker} \left \{\mathcal{H}_e^{1,0}/\mathcal{H}_{0,e}^{1,0} \stackrel{\bar{\theta}}{\longrightarrow} (\mathcal{H}_e^{1,0}/\mathcal{H}_{0,e}^{1,0})^{\perp} \otimes \Omega_S^1 \right \}$. Similar identities exists in higher weights. The above kind of identities are called Arakelov identities. The Arakelov identities measure the complexity of the global twist of the VHS, according to the existence of the degeneracies, see \cite{GGK} for details.

Another piece of information we encounter in our analysis is the Fujita decomposition of the Hodge bundle, \cite{CK, PT, Vi, VZ, CaDe, CaDe2}. Assume we have the fibration $f:X \to S$ as before [see Section 4 for exact set-up]. In general the Fujita decomposition is of the form
\begin{equation}
f_*\omega_{X/S}=\mathcal{A} \bigoplus \mathcal{U}' \bigoplus \mathcal{O}_S^{q_f}, \qquad 
\end{equation}

\noindent
where $\omega_{X/S}$ is the bundle of holomorphic differential forms $\Omega_X^{\dim X}$, and $\mathcal{A}$ is an  ample vector bundle, $\mathcal{U}=\mathcal{O}_S^{q_f} \bigoplus \mathcal{U}'$ is a flat unitary vector bundle. $\mathcal{U}'$ is also flat but has no sections. The Gauss-Manin connection on $\mathbb{H}$ induces a connection 

\begin{equation}
\nabla_{\mathcal{U}}: \mathcal{U} \to \mathcal{U} \otimes \Omega_S^1
\end{equation}

\noindent
It is known by construction that $\mathcal{U} \subset \ker(\theta)$. Let  
\begin{equation}
\kappa^{n,0} :\mathcal{H}^{n,0} \stackrel{\theta^n}{\rightarrow} \mathcal{H}^{n-1,1} \stackrel{\theta^{n-1}}{\rightarrow} ... \stackrel{\theta}{\rightarrow} \mathcal{H}^{0,n}
\end{equation} 
be the Griffiths-Yukawa coupling, i.e the composition of the $n$ successive KS-maps. 
The Griffiths-Yukawa coupling is said to be maximal, if $\mathcal{H}^{n,0} \ne 0$ and if $\kappa^{n,0}$ is an isomorphism, [cf. \cite{VZ4}]. Similar to before we set 
\begin{equation} 
\mathcal{H}_0^{n,0}=\ker \left ( \theta^p:\mathcal{H}^{p,q} \to \mathcal{H}^{p-1,q+1} \otimes \Omega_S^1(\log E) \right )
\end{equation} 
and $h_0^{n,0}=\text{rank} \mathcal{H}_0^{n,0}$, with the same notation as previous section. We have the Arakelov inequality
\begin{equation}
\deg(\mathcal{H}^{n,0}) \leq \frac{n}{2}\ . \text{rank}(\mathcal{H}^{n,0}). \deg(\Omega_S^1(\log E))
\end{equation}
with equality if the Higgs field of $\mathcal{H}$ is maximal. Moreover if $\theta^n \ne 0$ then $\deg(\mathcal{H}^{n,0}) >0 , \ \deg(\Omega_S^1(\log E)) >0$. In particular if $S=\mathbb{P}^1$ then $\sharp E \geq 3$. The Fujita decomposition $f_*\omega_{X/S}=\mathcal{A} \bigoplus \mathcal{U}$ into an ample sheaf $\mathcal{A}$ and a flat subsheaf $\mathcal{U}$, satisfies $\deg(\mathcal{H}^{n,0})=\deg(\mathcal{A})$. Moreover the Higgs field of $\mathcal{A}$ is strictly maximal and $\mathcal{U}$ is a variation of polarized complex Hodge structure zero in the bidegree $(n,0)$, [see \cite{VZ4} lemma 4, for instance]. 

\subsection{Contributions} 

We apply this formula for a VHS of surfaces $f:X \to S$ fibered by curves $Y_s$. One obtains a triangle fibration 

\begin{equation}
\xymatrix{
  & X \ar[dl]_{h} \ar[dr]^{f}
  \ar@{}[d]|-{\circlearrowleft} \\
  Y \ar[rr]_{g}
  && S
}
\end{equation} 

\noindent
Let $\mathcal{H}, \ \mathcal{H}(1)$ and $\mathcal{H}(2)$ be the associated local systems of $H^1$ of fibers in the triangle. Then an application of the Grothendieck 5-term sequence gives the following diagram of sheaves with Kodaira-Spencer maps 
\begin{equation}
\begin{CD}
g_*\Omega_{Y/S}^1 (\log F)@>>>  f_*\Omega_{X/S}^1(\log G) @>>> f_*\Omega_{X/Y}^1(\log D)   \\
@VV{\theta^{1}}V    @VV{\theta}V  @VVg_*{\theta^{2}}V \\
R^1g_*\mathcal{O}_Y \otimes \Omega_S^1(\log E) @>>> R^1f_*\mathcal{O}_X \otimes \Omega_S^1 (\log E) @>>> g_* (R^1h_*\mathcal{O}_X) \otimes \Omega_Y^1(\log E) \\  
\end{CD}
\end{equation}
\noindent
With $\log $ we mean logarithmic sheaves with appropriate normal crossing divisors, that assumed to be compatible in the fibration [see Section 2 for the set up and Definitions]. In the text, we sometimes denote the suffix $\theta_Y, \theta_{X/Y}$ to distinguish the KS-map. We study how the contributions in Arakelov equation (2) for the middle vertical KS-map splits in the other two. We discuss the following 

\begin{itemize}
\item There is a decomposition $\mathcal{H}=\mathcal{H}_{\text{fix}} \bigoplus \mathcal{H}_{\text{var}}$ into fixed and variable parts where 

\begin{equation} 
\mathcal{H}_{\text{var}}:=\{image(f_*\Omega_{X/S}^1 \to f_*\Omega_{X/Y}^1)\}^{\perp}
\end{equation} 

\noindent 
and $\mathcal{H}_{\text{fix}}$ is identified with the image. Then $g_*\left (\theta^{2}(\mathcal{H}_{var})\right )=0$

\item In the course to compare $\delta, \delta^{1}$, and $\delta^{2}$ there should be a relation 

\begin{equation}
\nu(\theta) = \nu(\theta^{1})+\nu(g_*\theta^{2}\mid_{\text{fix}})
\end{equation}

\noindent
where $\nu$ is computed on the points where KS-map drops rank, rather they are singularities or not.

\item The contribution for the terms $\delta_0$ ,s is additive
 
\begin{equation}
\delta_0=\delta_0^{1}+\delta_0(\mathcal{H}^{1,0}(2)_{\text{fix}})
\end{equation}

\item The $\mathcal{H}_{\text{var}}$ contribute to the similar invariants for the weight 2 VHS in $H^2(X_s, \mathbb{C})$. 
\end{itemize}

\vspace{0.3cm} 

\noindent 
We also compare the Fujita decomposition for the Hodge bundles of VHS over ${Y}$ and $S$. We have a similar decomposition

\begin{equation}
f_*\omega_{X/S}=f_*\omega_{Y/S} \oplus f_*\omega_{X/S}^{\text{var}}
\end{equation}

\noindent 
One has the estimates for the ranks of the direct summands of Fujita decomposition for the 3 fibrations. In a basic analysis we get that 
\begin{itemize}
\item $H_e^{1,0} =H_e^{1,0}(1) \bigoplus H_e^{1,0}(2)_{\text{fix}}$.
\item $\mathcal{U}_f=\mathcal{U}_g \bigoplus \mathcal{U}_h^{\text{fix}}$.
\item $h^{1,0}=h(1)^{1,0}+h_{\text{fix}}^{1,0}(2)$. 
\item $h_0^{1,0}=h_0^{1,0}(1)+h_{0,\text{fix}}^{1,0}(2)$.
\item $\delta_0 = \delta_0^{1}+\left (\delta_0^{2}\right )_{\text{fix}}$.
\item $\nu(\theta)= \nu(\theta^{1})+\nu(g_*\theta^{2}\mid_{\text{fix}})$.
\item $\delta^{1,0}=\delta^{1}+\delta^{2}_{\text{fix}}=\delta^{1}+\delta^{2}-\delta^{2}_{\text{var}}$.
\item $u_f=u_g+u_h^{\text{fix}}=u_g+u_h-u_h^{\text{var}}$.
\end{itemize}

\noindent 
By the local product structure of the double fibration in (8) and Kodaira-Spencer maps induced from the Gauss-Manin connections we deduce the existence of an induced map 

\begin{equation}
\overline{\vartheta^{2,0}}: \dfrac{g_*\left (R^1h_*\Omega_{X/Y,var}^0\right )}{\theta^{2}_{\text{var}} \left ( g_* (R^0h_*\Omega_{X/Y,var}^1) \right )} \ \longrightarrow \ \dfrac{R^2f_*\Omega_{X/S}^0 \otimes (\Omega_S^1)^{\otimes 2}}{\theta^{1,1}\left ( R^1f_*\Omega_{X/S}^1 \otimes \Omega_S^1 \right )}
\end{equation}

\noindent
which is injective. We shall use the existence of this map to obtain Arakelov type identities in a two step fibration. We prove the following identity on the degrees of the Hodge bundles in a family of surfaces fibered by curves

\begin{equation}
\delta^{2,0}+2(\delta^{1,0}-(\delta^{1}+\delta^{2}) \geq (h^{1,1}-h^{2,0})(2g-2)^2-h^{1,0}(2)(2g-2)
\end{equation}

\noindent
By adjusting the existence of a local product structure on the fibred surface we can obtain another inequality which relates the the Hodge numbers in weight two HS to the one of weight one on the middle curves. Another inequality we prove is

\begin{equation}
\delta_X^{2,0}\geq \frac{1}{2}(h_{X/Y}^{1,0}-h_{0,X/Y}^{1,0})(2g(S)-2)+ h_{X/Y}^{1,0}
\end{equation}

\noindent
Moreover we have $\delta_X^{2,0}\ \geq \ \delta_{X/Y}^{1,0}$. The VHS of the middle cohomology of the fibration $Y \to S$ appears as a middle convolution of the one for $X \to S$, \cite{SD}. We give the following generalization of a result in \cite{SD} on the invariants of the middle convolution of the local systems of HS of surfaces. The result there was proved over $S=\mathbb{P}^1$ which is more simpler. Our statement is as follows. If the local system $\mathcal{V}=g_* \left (R^1h_*\mathbb{C} \right )$ is irreducible with regular singularities, then 

\begin{equation}
\chi(S, \mathcal{V})=\delta_{X/Y}^1+h^1(V)-\delta_{X/Y}^0+h^0(V)-\sum_s\nu_s(\bar{\theta}_{X/Y})(V).
\end{equation}

\vspace{0.5cm}

\noindent 
\textbf{Organization:} Section 2 is classical and we just introduce the notation and concepts we wish to employ. References are \cite{Har, G3, PS, JS1, JS2, JS3}. Section 3 describes numerical invariants of variation of Hodge structures and closely follows \cite{GGK}. The relevant references are \cite{Sch1, PS, Vi, St1, Hod, AGV, B, D1, G1, KUL, SC2, SCHU, JS1, JS2, JS3}. Section 4 explains Fujita decomposition. The relevant references are \cite{G, CK, PT, Kaw1, Fuj1, Fuj2, CaDe, CaDe2}. In section 5 we present some contributions and calculations by the numerical invariants.  

\section{Algebraic geometric background}

We briefly review some basics on relative logarithmic differentials on schemes used in later sections. This section is classical, and the reader can find various texts and references describing logarithmic sheaves in slightly different contexts (see \cite{G3, PS, JS1, JS2, JS3}). Lets begin with a smooth fibration of schemes $f:X \to S$ where $X$ and $S$ are smooth and of finite type over an algebraically closed field of char=0. In this case the module of relative q-differentials of $f$ on $X$ 
is defined by the short exact

\begin{equation}
0 \to f^*\Omega_S^1 \wedge \Omega_X^{q-1} \to \Omega_X^q \to \Omega_{X/S}^q \to 0
\end{equation}

\noindent
When you have three schemes 

\begin{equation} 
f:X \stackrel{h}{\rightarrow} Y \stackrel{g}{\rightarrow} Z
\end{equation} 

\noindent
respectively, then we have the following exact sequence

\begin{equation}
0 \to h^*\Omega_{Y/S}^1 \wedge \Omega_{X/S}^{q-1} \to \Omega_{X/S}^q \to \Omega_{X/Y}^q \to 0
\end{equation}

In case $X$ or $S$ not to be regular or $f$ fails to be smooth the above definition is no longer valid. However one still may define it when the singular locus is normal crossing, called logarithmic differentials. In this case we replace our schemes with log-pairs $(X,D)$ and $(S,E)$, where $D$ and $E$ are smooth normal crossing schemes or varieties and $f(D) =E$. According to the Hironaka resolution of singularities theorem this is possible after blowing up the singularities of $X, S$ along $f$ enough times. In this case one can similarly choose  

\begin{equation}
\Omega^q_{X/S}(\log D)=\frac{\Omega^q_{X}(\log D)}{f^*\Omega^1_{S}(\log E) \wedge \Omega^{q-1}_{X}(\log D)}
\end{equation} 

\noindent
as the definition, \cite{PS, JS1, JS2, JS3}. We shall consider three log-pairs 

\begin{equation} 
f: (X,D) \stackrel{h}{\rightarrow} (Y,F) \stackrel{g}{\rightarrow} (S,E)
\end{equation}  

\noindent 
Then we are assuming they are compatible log-smooth pairs in the category of log-schemes. Then the sequence (20) is valid in the new category and by definition reads as

\begin{equation}
0 \to h^*\Omega_{Y/S}^1(\log F) \wedge \Omega_{X/S}^{q-1}(\log D) \to \Omega_{X/S}^q (\log D)\to \Omega_{X/Y}^q (\log h^* F) \to 0
\end{equation} 

\noindent
We will be working with 1-forms and the global sections of the stalk of these sheaves along fibers over $S$. In other word when pushforward by $f_*$ we get

\vspace{0.5cm}

\begin{equation}
\begin{aligned}
0 \to g_*\Omega_{Y/S}^1(\log F) \to f_*\Omega_{X/S}^1 (\log D) & \to f_*\Omega_{X/Y}^1 (\log h^* F) \to R^1g_*\Omega_{Y/S}^1(\log F) 
\to \\
& \to R^1f_*\Omega_{X/S}^1 (\log D)\to R^1f_*\Omega_{X/Y}^1 (\log h^* F)
\end{aligned}
\end{equation}

\noindent
Note that varieties that the items in this exact sequence refer to are the appropriate fibers in our triangle and the normal crossing divisors
give normal crossing sections inside the fibers. Regarding the situation that how the relative singular fibers of the maps $f,g$ and $h$ behave geometrically relative to each other, we hereby and later on are assuming that they are located one above another so that the relative log-pairs are compatible as mentioned above [cf. \cite{PS, JS1, JS2, JS3} loc. cit.]. 

Another fact, is that going to the long exact cohomology sequence of the short exact sequence defining the relative differentials we obtain

\begin{equation}
0 \to \Omega_S^1 \to f_*\Omega_X^1 \to f_*\Omega_{X/S}^1 \stackrel{\kappa=(.\kappa_s)}{\longrightarrow} R^1(f_*f^*)\Omega_S^1=R^1f_*\mathcal{O}_X \otimes \Omega_S^1 \to
\end{equation}

\noindent
Then the connecting homomorphism $\kappa$ is the Kodaira-Spencer map, i.e wedge with the Kodaira-Spencer class. For the last term in the sequence we have used adjunction formula 

\begin{equation}
\mathcal{G} \otimes R^if_* \mathcal{E} =R^if_*f^*\mathcal{G} \otimes \mathcal{E}\end{equation} 

\noindent 
for suitable sheaves $\mathcal{G}, \ \mathcal{E}$. This phenomenon maybe studied in a triangle fibration of smooth varieties, or a compatible log-smooth triple.

We use a basic property of Grothendieck spectral sequences, applied to a first quadrant spectral sequence 

\begin{equation}
E_2^{p,q}=R^pG \circ R^qF (A) \Rightarrow R^{p+q}(G \circ F)(A)
\end{equation}

\noindent
for functors $F:\mathcal{A} \to \mathcal{B}, \ G:\mathcal{B} \to \mathcal{C}$ between abelian categories. It is an exact sequence of terms of low degrees.

\begin{proposition} (Grothendieck 5-term sequence)
If $E_2^{p,q}  \Rightarrow H^n(A)$ is a first quadrant spectral sequence, then there is an exact sequence 

\begin{equation}
0 \to E_2^{1,0} \to H^1(A) \to E_2^{0,1} \to E_2^{2,0} \to H^2(A)
\end{equation}

\noindent
where $E_2^{0,1} \to E_2^{2,0}$ is the differential of the $E_2$-term.
\end{proposition}

\noindent 
When we have two successive fibrations $X \to Y \to S$ we get a form of 5-term sequence as 

\begin{equation}
0 \to R^1g_*\mathcal{O}_Y  \to R^1f_*\mathcal{O}_X \to g_*R^1h_*\mathcal{O}_X  \to R^2g_*\mathcal{O}_Y  \to R^2f_*\mathcal{O}_X 
\end{equation}

\noindent
We will employ a twist of this sequence in computation with Kodaira-Spencer maps.

\section{Hodge-Arakelov numerical data}

The reference for this section is \cite{GGK}, \cite{JS2}. We follow \cite{GGK} briefly to adjust the formulas we use. Assume $f:X \to S$ is a fibration of smooth projective varieties over $\mathbb{C}$. Hodge theory studies the variation of Hodge structure (VHS)constructed from the cohomologies of the smooth fibers $X_s=f^{-1}(s)$. The VHS,s associated to these fibrations form a local system $\mathcal{H}_{/S^*}$ of vector spaces over $\mathbb{Q}$, such that 

\begin{equation}
\mathcal{H}_{\mathbb{C}}=\bigoplus_{p+q=k} \  \mathcal{H}^{p,q}, \qquad \mathcal{H}^{p,q}:=F^p \cap \overline{F^q}
\end{equation} 

\noindent 
are the  ($C^{\infty}$)-Hodge bundles defined over the smooth locus $S^*$. By the Hironaka resolution of singularities one may assume the degeneracy locus $E=S \setminus S^*$ is normal crossing divisor. The classical Hodge theory guarantees the existence of a canonical extension $\mathcal{H}_e$ whose sections have at worst logarithmic poles along the normal crossing divisor $E$. A major task in Hodge theory is to study the degeneration of the Hodge structure near the singular points. Assume $D=f^{-1}(E)$ is also normal crossing. We may draw the picture of our fibration as follows

\vspace{0.3cm}

\begin{center}
$\begin{CD}
D @>>> X @<<<  X\setminus D  \\
@VVV    @VVV  @VVV \\
E @>>> S @<<< S\setminus E 
\end{CD}$
\end{center} 

\vspace{0.3cm}

\noindent
Then the sheaves $R^pf_*\Omega^\bullet_{X/S}(\log D)$ are locally free, and admit a (logarithmic) connection

\begin{equation}
\nabla_e:R^pf_*\Omega^\bullet_{X/S}(\log D) \to \Omega^1_{S}(\log E) \otimes R^pf_*\Omega^\bullet_{X/S}(\log D)
\end{equation} 

\noindent
such that its residue is nilpotent. Moreover there are isomorphisms 

\begin{equation}
H^p(D,\Omega^\bullet_{X/S}(\log D)) \otimes \mathcal{O}_D) \stackrel{\cong}{\longrightarrow} H^p(X_{\infty},\mathbb{C})
\end{equation}

\noindent
Among the invariants associated to VHS,s are the degrees of the Hodge bundles 

\begin{equation}
\delta_p=\deg(F_e^{n-p}/F_e^{n-p+1})
\end{equation}

\noindent
The better are the quantities 

\begin{equation}
\delta_0+\delta_1+...+\delta_p \geq 0
\end{equation}

\noindent 
which measure how far the VHS is being from trivial. Over the smooth locus the Hodge bundles $\mathcal{H}^{p,q}$ are equipped with the Hodge metric induced from the polarization form. The invariants $\delta_p$ may be calculated from the curvature of the associated metric connection. The curvature can be written as

\begin{equation}
\Theta_{\mathcal{H}^{p,q}}=^t\theta_q \wedge \bar{\theta}_q + \bar{\theta}_{q-1} \wedge ^t\theta_{q-1}  
\end{equation}

\noindent
where $\theta_q=Gr_F^p \nabla$ are the Kodaira-Spencer maps and "$^t$" is the Hermitian adjoint. Here we understand that $q=n-p$. The cohomology classes 

\begin{align}
\begin{aligned}
c_1(\Theta_{\mathcal{H}^{p,q}})=\alpha_p-\alpha_{p-1}, \qquad \int_S \alpha_p=a_p\\
c_1(\Theta_{\mathcal{H}_0^{p,q}})=\beta_p-\alpha_{p-1}, \qquad \int_S \beta_p=b_p
\end{aligned}
\end{align}

\noindent
are integrable closed (1,1)-forms and also determine the Chern classes of the extended Hodge bundle; $\mathcal{H}_e^{p,q}$.  

In the weight one VHS,s obtained from the fibration by curves we have the following exact sequence of sheaves

\begin{equation}
0 \to \mathcal{H}_{0,e}^{1,0} \to \mathcal{H}_e^{1,0} \stackrel{\theta^{(1,0)}}{\longrightarrow} \check{\mathcal{H}}_e^{1,0} \otimes \Omega_S^1(\log E) \to \text{coker} \to 0
\end{equation}

\noindent
where $\mathcal{H}_{0,e}^{1,0}:=\ker({\mathcal{H}}_e^{1,0} \to \check{\mathcal{H}}_e^{1,0} \otimes \Omega_S^1)$. We denote the rank of its generic fiber by $h_0^{1,0}$. Note that $\mathcal{B}$ is probably not a vector bundle; i.e may have torsion. It is also probable that the meromorphic $h^{1,0} \times h^{1,0}$-matrix $\theta$ drops rank at non singular values $s$. Because the degree map is additive on exact sequences one can calculate the $\delta=\deg(\mathcal{H}^{p,q})$ in terms of the other terms. For the map on the stalks we have 

\begin{equation}
\theta_s(\mathcal{H}_{e,s}^{1,0}) \subset (\mathcal{H}_{0,e,s}^{1,0})^{\perp} \otimes \Omega_{S,s}(\log E)
\end{equation}

\noindent
One considers set of points where this inclusion is strict. This may also happen at nonsingular points. The calculation of the curvature gives

\begin{equation}
\delta=\frac{1}{2}(h^{1,0}-h_0^{1,0})(2g-2)+\delta_0-\frac{1}{2}\sum_{s \in S} \nu_s(\bar{\theta})
\end{equation}

\noindent
where $\delta_0=\deg(\mathcal{H}_{0,e}^{1,0}), \ \nu_s(\theta):=\dim (\text{coker}(\theta))$ and 

\begin{equation} 
\bar{\theta}: \mathcal{H}_e^{1,0}/\mathcal{H}_{0,e}^{1,0} \longrightarrow (\mathcal{H}_e^{1,0}/\mathcal{H}_{0,e}^{1,0})^{\vee} \otimes \Omega_S^1
\end{equation}

\noindent
Alternatively (39) can be written as

\begin{equation}
\begin{split}
\delta & =\frac{1}{2}\Big [ (h^{1,0}-h_0^{1,0})(2g-2)+\sum_i \dim({Im}\bar{N}_i) \Big ] - \\
& \qquad \qquad \qquad \qquad \qquad \Big [ (-\delta_0)+\frac{1}{2}\left ( \sum_{s \in S^*} \nu_s (\bar{\theta}_s)+\sum_{i}\nu_{s_i} (\bar{A}_i) \right ) \Big ]
\end{split}
\end{equation} 
\noindent 
where $N_i$ are the logarithms of the monodromies at the degeneracy point 
$s_i$ and since $N_i=\text{Res}_{s_i}\nabla$,

\begin{equation}
\bar{N_i}:  F_{e,s_i}^1/F_{0,e,s_i}^{1} \longrightarrow \left ( F_{e,s_i}^{1}/F_{0,e,s_i}^{1}\right )^{\vee}
\end{equation}

\noindent
are the induced map. The matrices $A_i$ are defined via the matrice of the map $\bar{\theta}_{s_i}$ by

\begin{equation}
\bar{\theta}_{s_i}= \begin{pmatrix} 
\frac{B'}{s}+... & * &*\\
0 &A'(s) & *\\
0 & 0 & 0 
\end{pmatrix}
\end{equation} 

\noindent 
where $B'$ is the nonzero component of the matrix of $\theta_{s_i}$, \cite{GGK}. 

In the weight 2 case in the absence of degeneracy, we have the two short exact sequence

\begin{align}
\begin{aligned}
0 \to \mathcal{H}_{0,e}^{2,0} \to \mathcal{H}_e^{2,0} \stackrel{\theta^{2,0}}{\longrightarrow} {\mathcal{H}}_e^{1,1} \otimes \Omega_S^1 \stackrel{\theta_1}{\rightarrow} \check{\mathcal{H}}_{0,e}^{1,1} \otimes \Omega_S^1 \to 0 \\
0 \to \mathcal{H}_{0,e}^{1,1} \to \mathcal{H}_e^{1,1} \stackrel{(\theta^{2,0})^{\vee}}{\longrightarrow} {\mathcal{H}}_e^{0,2} \otimes \Omega_S^1 \to \check{\mathcal{H}}_e^{2,0} \otimes \Omega_S^1 \to 0
\end{aligned}
\end{align}

\noindent
which are dual of one another. The equation (39) is replaced by the following

\begin{equation}
\delta^{2,0}=(h^{2,0}-h_0^{2,0})(2g-2)-[\delta(0)+\sum_{s \in S} \nu_s(\bar{\theta}^{2,0})]
\end{equation}

\noindent
where 

\begin{equation}
{\bar{\theta}^{2,0}}:\mathcal{H}^{2,0}/\mathcal{H}_{0}^{2,0} {\longrightarrow} \mathcal{H}^{1,1} \otimes \Omega_S^1/\ker \theta_1
\end{equation}

\noindent 
and $\delta(0)=a_0+b_0+b_1$. When there are degeneracies the formula modifies as follows

\begin{equation}
\delta^{2,0}=(h^{2,0}-h_0^{2,0})(2g-2+N)+[\delta(0)+ \sum_{s \in S} \nu_s(\det \theta)+\sum_i(\hat{h}_i^{2,0}-h_i^{2,0})]
\end{equation}

\noindent
where we have set $\hat{H}_i=Gr_q(LMHS)_{s_i}$ and $\hat{h}_i^{p,q}$ are their Hodge numbers, cf. \cite{GGK} loc. cit.. 

\section{Fujita decomposition}

Fujita decomposition associated to a fibration $f:X \to S$ of smooth projective schemes $X, S$ over $\mathbb{C}$ and where $\dim(S)=1, \ \dim(X)=n+1$ is a decomposition of the sheaf $f_*\omega_{X/S}$ where $\omega_{X/S}$ is the relative canonical sheaf of $f$. First denote 

\begin{equation}
f:X^* \to S^*, \qquad X^*=X \smallsetminus \text{Sing}(f)
\end{equation}

\noindent
to be the same map on the smooth locus of $S$. Then the direct image $\mathcal{H}=R^nf_*\mathbb{C}_{X^*}$ defines a local system on $S^*$. By Riemann-Hilbert correspondence it is equivalent to the Gauss-Manin connection

\begin{equation}
\nabla:\mathcal{H} \otimes \mathcal{O}_S \to \mathcal{H} \otimes \Omega_S^1
\end{equation}

\noindent
Its data can be encoded in a monodromy representation

\begin{equation}
\varrho:\pi_1(S^*) \to GL(\mathcal{H}_{s_0})
\end{equation} 

\noindent
where $s_0 \in S$. When studying this local system globally there appear two type of monodromy representations. One is the monodromy around singularities and second is the monodromy of $S$.

\begin{theorem} (Fujita, Catanese-Dettweiler) \cite{PT, Fuj1, Fuj2} There is a decomposition 

\begin{equation}
f_*\omega_{X/S}=\mathcal{A} \bigoplus \mathcal{U}, \qquad \text{(2nd Fujita decomposition)}
\end{equation}

\noindent
where $\mathcal{A}$ is an ample vector bundle on $S$, and $\mathcal{U}$ is unitary flat sheaf whose restriction to $S^*$ is a holomorphic vector bundle associated to a sub-local system of $\mathcal{H}$. 
\end{theorem}

\noindent
There is also a first Fujita decomposition.

\begin{theorem}(Fujita) \cite{PT, Fuj1, Fuj2}
There exists a decomposition 

\begin{equation}
f_*\omega_{X/S}=\mathcal{E} \bigoplus \mathcal{O}_S^{q_f}, \qquad \text{(1nd Fujita decomposition)}
\end{equation}

\noindent
with $\mathcal{E}$ a vector bundle, and $q_f=h^1(\mathcal{O}_X)-g(S)=q(X)-g(S)$ is the relative irregularity of $f$.
\end{theorem}

\noindent
The embedding $\mathcal{O}_S^{q_f} \hookrightarrow \mathcal{U}$ splits as

\begin{equation}
\mathcal{U}=\mathcal{O}_S^{q_f} \bigoplus \mathcal{U}'
\end{equation}

\noindent
where $\mathcal{U}'$ is also flat but has no sections. The decomposition (52) is compatible with that in (51), that we have

\begin{equation}
f_*\omega_{X/S}=\mathcal{A} \bigoplus \mathcal{U}' \bigoplus \mathcal{O}_S^{q_f} 
\end{equation}

\noindent
The Gauss-Manin connection on $\mathcal{H}$ induces a connection 

\begin{equation}
\nabla_{\mathcal{U}}: \mathcal{U} \to \mathcal{U} \otimes \Omega_S^1
\end{equation}

\noindent
We denote the corresponding local system by $\mathbb{U}$. We have $\mathcal{U}=\mathbb{U} \otimes \mathcal{O}_S$. The data of the local system is equivalent to a unitary representation

\begin{equation}
\tau:\pi_1(S) \to U(q_f,\mathbb{C})
\end{equation} 

\noindent
That is $\mathbb{U}$ corresponds to the global monodromy of $S$. Fujita decompositions can also be stated over a higher dimensional base, \cite{CK}. It is known by construction that $\mathcal{U} \subset \ker(\theta)$. It follows that

\begin{equation}
u_f:=\text{rank}(\mathcal{U}) \leq h_0^{n,0}
\end{equation}

\noindent 
Assume $f:X \to S$ is a fibration of $n$-folds over a smooth curve $S$ and $\mathcal{H}=R^nf_*\mathbb{C}=\bigoplus_{p+q=n} \mathcal{H}^{p,q}$ is a variation of polarized complex Hodge structure of weight $n$ with unipotent monodromy over $S$. Let  
\begin{equation}
\kappa^{n,0} :\mathcal{H}^{n,0} \stackrel{\theta^n}{\rightarrow} \mathcal{H}^{n-1,1} \stackrel{\theta^{n-1}}{\rightarrow} ... \stackrel{\theta}{\rightarrow} \mathcal{H}^{0,n}
\end{equation} 
be the Griffiths-Yukawa coupling, i.e the composition of the $n$ successive KS-maps. 
The Griffiths-Yukawa coupling is said to be maximal, if $\mathcal{H}^{n,0} \ne 0$ and if $\kappa^{n,0}$ is an isomorphism. The coupling is said to be strictly maximal if the same conditions hold but all the map stages $\theta^i$ in (56) are isomorphisms, [cf. \cite{VZ4}]. Similar to before we set $\mathcal{H}_0^{n,0}=\ker \left ( \theta^p:\mathcal{H}^{p,q} \to \mathcal{H}^{p-1,q+1} \otimes \Omega_S^1(\log E) \right )$ and $h_0^{n,0}=\text{rank} \mathcal{H}_0^{n,0}$, with the same notation as previous section. 

\begin{theorem} [\cite{VZ4} proposition 1.1]
We have the Arakelov inequality
\begin{equation}
\deg(\mathcal{H}^{n,0}) \leq \frac{n}{2}\ . \text{rank}(\mathcal{H}^{n,0}). \deg(\Omega_S^1(\log E))
\end{equation}
with equality if the Higgs field of $\mathcal{H}$ is maximal. Moreover if $\theta^n \ne 0$ then $\deg(\mathcal{H}^{n,0}) >0 , \ \deg(\Omega_S^1(\log E)) >0$. In particular if $S=\mathbb{P}^1$ then $\sharp E \geq 3$. 

Assume $A^{n,0}$ be a subsheaf of $\mathcal{H}^{n,0}$ such that we have the equality above when $\mathcal{H}^{n,0}$ is replaced by $A^{n,0}$. Then we have a decomposition $\mathcal{H} =\mathcal{A} \bigoplus \mathcal{U}$ and subsheaves $A^{p,n-p}$ of $\mathcal{H}^{p,n-p}$ such that the Higgs bundle of $\mathcal{A}$ is $(\bigoplus A^{p,q}, \bigoplus \tau^p)$ and the Higgs field $\tau$ is strictly maximal. 
\end{theorem} 
\noindent 
The Fujita decomposition $f_*\omega_{X/S}=\mathcal{A} \bigoplus \mathcal{U}$ into an ample sheaf $\mathcal{A}$ and a flat subsheaf $\mathcal{U}$, satisfies 
\begin{equation}
\deg(\mathcal{H}^{n,0})=\deg(\mathcal{A})
\end{equation}  
Moreover the Higgs field of $\mathcal{A}$ is strictly maximal and $\mathcal{U}$ is a variation of polarized complex Hodge structure zero in the bidegree $(n,0)$, [see \cite{VZ4} lemma 4, for instance]. 
 
\section{Family of Surfaces fibered by Curves}

We consider a system of fibration of surfaces over families of curves as 

\begin{equation}
\xymatrix{
  & X \ar[dl]_{h} \ar[dr]^{f}
  \ar@{}[d]|-{\circlearrowleft} \\
  Y \ar[rr]_{g}
  && S}
\end{equation}

\noindent
One recovers the Leray-Serre spectral sequence from above. Apply the Grothendieck spectral sequence to the following spectral sequence 

\begin{equation}
R^pg_*(R^qh_*\mathcal{F}) \Rightarrow R^nf_*\mathcal{F}
\end{equation}

\noindent
where $\mathcal{F} \in Coh(X)$. It is obviously convergent. One obtains 

\begin{equation}
0 \to R^1g_*\mathcal{O}_Y  \to R^1f_*\mathcal{O}_X \to g_*R^1h_*\mathcal{O}_X  \to R^2g_*\mathcal{O}_Y  \to R^2f_*\mathcal{O}_X 
\end{equation}

\noindent
Applying $\otimes \Omega_S^1$ we get the following

\begin{multline}
0 \to R^1g_*\mathcal{O}_Y \otimes \Omega_S^1 \to R^1f_*\mathcal{O}_X \otimes \Omega_S^1 \to g_*R^1h_*\mathcal{O}_X \otimes \Omega_S^1\\  \to R^2g_*\mathcal{O}_Y \otimes \Omega_S^1 \to R^2f_*\mathcal{O}_X \otimes \Omega_S^1
\end{multline}

\noindent
This sequence fits with the long exact cohomology sequence of 

\begin{equation}
0 \to g_*\Omega_{Y/S}^1 \to  f_*\Omega_{X/S}^1 \to f_*\Omega_{X/Y} \to 0 
\end{equation}

\noindent
and gives

\begin{proposition}
In the absence of singularities (degenerations) we have a commutative diagram 

\begin{center}
$\begin{CD}
@. \Omega_{S}^1 @.  \Omega_{S}^1 @. \Omega_{Y}^1 @.  \\
@.  @VVV    @VVV  @VVV \\
@. g_*\Omega_{Y}^1 @.  f_*\Omega_{X}^1 @. f_*\Omega_{X}^1 @.  \\
@.  @VVV    @VVV  @VVV \\
0 @>>> g_*\Omega_{Y/S}^1 @>>>  f_*\Omega_{X/S}^1 @>>> f_*\omega_{X/Y} @>>> R^1g_*\Omega_{Y/S}^1 \\
@.  @VV{\theta_Y}V    @VV{\theta_X}V  @VVg_*{\theta_{X/Y}}V \\
0 @>>> R^1g_*\mathcal{O}_Y \otimes \Omega_S^1 @>>> R^1f_*\mathcal{O}_X \otimes \Omega_S^1 @>>> g_*R^1h_*\mathcal{O}_X \otimes \Omega_S^1  @>>> R^2g_*\mathcal{O}_Y\\
@.  @VVV    @VVV  @VVV  
\end{CD}$
\end{center}

\end{proposition}

\begin{proof} We have the short exact sequence downstairs because $R^2g_*\mathcal{O}_Y=0$ (Kodaira vanishing). The fact that the connecting homomorphism in the vertical directions is the Kodaira-Spencer map is a classical known fact. The only thing that remains to be proved is that
the third term in the horizontal row upstairs is well defined and correct. First note that in our case assuming that no singularity may appear in the fibers of $X \to S$ nor $Y \to S$, it is not guaranteed that the fibers of $X_s \to Y_s$ are all non singular. In general the  exact sequence must be written in the form

\begin{equation}
0 \to g_*\Omega_{Y/S}^1 \to  f_*\Omega_{X/S}^1 \to f_*\Omega_{X/Y}^1(\log D) \to R^1g_*\Omega_{Y/S}^1 \\
\end{equation}

\noindent
where $D$ is normal crossing. However $\Omega_{X/Y}^1(\log D)$ is a line bundle, i.e of rank one. Therefore

\begin{equation}
\begin{aligned}
\Omega_{X/Y}^1(\log D) &= \det (\Omega_{X/Y}^1(\log D))\\ 
&= \det (\Omega_{X}^1(\log D)) \otimes \det (h^*\Omega_{Y}^1(\log F))^{\vee}\\   & =\omega_{X/Y}
\end{aligned}
\end{equation}

\noindent
The stalks of the horizontal row upstairs are

\begin{equation}
H^0(Y_s, \Omega_{Y_s}^1) \to H^0(X_s, \Omega_{X_s}^1) \to H^0(X_s, \omega_{X/Y}|_{X_s})
\end{equation}

\noindent
For the last term we have 

\begin{equation}
H^0(X_s, \omega_{X/Y}|_{X_s})=H^0(Y_s, h_*\omega_{X_s/Y_s})=H^0(W_t,\Omega_{W_t}^1)
\end{equation}

\noindent
where $W_t$ are relative fibers in $X_s \to Y_s$. The proposition follows. 
\end{proof}

\vspace{0.3cm}

\begin{theorem}
If we have degeneracies then the aforementioned digram modifies as 

\begin{center}
$\begin{CD}
 g_*\Omega_{Y}^1 @.  f_*\Omega_{X}^1 @. f_*\Omega_{X}^1 @.  \\
  @VVV    @VVV  @VVV \\
 g_*\Omega_{Y/S}^1(\log F) @>>>  f_*\Omega_{X/S}^1(\log G) @>>> f_*\omega_{X/Y} (\log D) @>>>  \\
  @VV{\theta_Y}V    @VV{\theta_X}V  @VVg_*\theta_{X/Y}V \\
 R^1g_*\mathcal{O}_Y \otimes \Omega_S^1(\log E) @>>> R^1f_*\mathcal{O}_X \otimes \Omega_S^1(\log E) @>>> g_*R^1h_*\mathcal{O}_X \otimes \Omega_S^1 (\log E) @>>> \\
  @VVV    @VVV  @VVV  
\end{CD}$
\end{center}

\vspace{0.3cm}

\noindent
where $D$ is a union of curves or points.
\end{theorem}
\begin{proof}
We set the singular locus to be $E, F, G$ in $S, Y, X$ respectively and assume we have sufficiently blown up that all being normal crossing and no multiplicities. We shall assume all the monodromies being unipotent.

\begin{equation}
\begin{CD}
G @>>> F @>>> E\\
@VVV @VVV @VVV\\
X @>>> Y @>>> S
\end{CD}
\end{equation}

\noindent
We have the following exact sequence of logarithmic sheaves

\begin{equation}
0 \to g_*\Omega_{Y/S}^1(\log F) \to  f_*\Omega_{X/S}^1(\log G) \to f_*\Omega_{X/Y}^1(\log D) \to R^1g_*\Omega_{Y/S}^1(\log F)\\
\end{equation}

\noindent
where 

\begin{equation} 
D=\bigcup_{t \in Y}W_t \qquad (h^{-1}(t)\ \text{is Singular} )
\end{equation}

\noindent  
Lets for simplicity consider the case where no degeneracies appear in the fibrations $X \to S$ and $Y \to S$. Working locally over $s \in S$, we can choose coordinates as 

\vspace{0.2cm}

\begin{center}
$x=(v,u,s) \mapsto y=(u,s) \mapsto s$ 
\end{center}

\vspace{0.2cm}

\noindent
in the fibrations. Then $p \in W_t$ is singular if and only if the matrix

\begin{equation}
Dh(p)= \left( 
\begin{array}{cccc}
\partial h_s/\partial x(p)  & \partial h_s/\partial u(p) & \partial h_s/\partial s(p) \\
 0 & 0  &  1 
\end{array} \right)
\end{equation}

\noindent
drops rank, i.e. has rank $\leq 1$. Thus $D$ is closed in $X$. Also its image in $Y$ is a closed proper subvariety, i.e. a union of curves and isolated points. 
\end{proof}

\begin{lemma} [cf. \cite{GG2} page 286]
There is a decomposition 

\begin{equation}
\mathcal{H}:=f_*\omega_{X/S}=\mathcal{H}_{\text{fix}} \bigoplus \mathcal{H}_{\text{var}}
\end{equation}

\noindent
into fixed and variable parts where  

\begin{align}
\begin{aligned}
\mathcal{H}_{\text{var}}:&=\{image(f_*\Omega_{Y/S}^1 \to f_*\omega_{X/S})\}^{\perp}\\
\mathcal{H}_{\text{fix}}:&=\{image(f_*\Omega_{Y/S}^1 \to f_*\omega_{X/S})\}
\end{aligned}
\end{align}
\end{lemma}
\begin{proof}
We can illustrate the decomposition $H^1(X_s, \mathbb{C})=H^1(Y_s, \mathbb{C}) \oplus H_{\text{var}}^1(X_s, \mathbb{C})$ where

\begin{equation}
H_{\text{var}}^1(X_s, \mathbb{C})=\text{im} \{H^1(X_s, \mathbb{C}) \to H^1(X_s, \mathbb{C})\}^{\perp}
\end{equation}

\noindent
For $s \in S$ the VHS of $H^1(X_s, \mathbb{C})$ is a sub-HS of the one for $H^1(X_s, \mathbb{C})$. Then by the semisimplicity of the monodromy, it has a complement that is also invariant by the monodromy action. Therefore we have 

\begin{equation}
\begin{aligned}
R^1f_*\mathbb{C}&=R^1g_*\mathbb{C} \oplus \left ( R^1f_*\mathbb{C}\right )^{\text{var}}\\
R^1f_*\mathcal{O}_X&=R^1g_*\mathcal{O}_Y \oplus \left ( R^1f_*\mathcal{O}_X \right )^{\text{var}}
\end{aligned}
\end{equation}

\noindent
This proves the lemma.
\end{proof}

\noindent
In surfaces fibered by curves one may consider the Fujita decompositions for the 3 fibrations involved. If we have a fibration of surfaces fibered by curves as in (8) then each of the 3 fibrations $h,\ g,\ f$ produce 3 Fujita decompositions as in (54). Thus we have 3 decompositions 

\begin{equation}
\begin{aligned}
f_*\omega_{X/S}&=\mathcal{A}_X \bigoplus \mathcal{U}_X' \bigoplus \mathcal{O}_S^{q_f},\\
f_*\omega_{Y/S}&=\mathcal{A}_Y \bigoplus \mathcal{U}_Y' \bigoplus \mathcal{O}_S^{q_g},\\
f_*\omega_{X/Y}&=\mathcal{A}_{X/Y} \bigoplus \mathcal{U}_{X/Y}' \bigoplus \mathcal{O}_S^{q_h}   
\end{aligned} 
\end{equation}

\noindent
where we have used a version of Theorem 4.1 and 4.2 over a higher dimensional base in the third identity, cf. \cite{CK}. We investigate the relation between unitary sheaves $\mathcal{U}_f',\ \mathcal{U}_g',\ \mathcal{U}_h'$ and their ranks. In the same way one can  prove the existence of a decomposition

\begin{equation}
f_*\omega_{X/S}=f_*\omega_{Y/S} \oplus f_*\omega_{X/S}^{\text{var}}
\end{equation}

\noindent
where 

\begin{equation}
f_*\omega_{X/S}^{\text{var}}=\{image \left (f_*\omega_{Y/S} \longrightarrow f_*\omega_{X/S}\right ) \}^{\perp}
\end{equation} 
 
\noindent
Assume for the moment we deduce that there is an image of the second row in the first row. Because the component $A$ is a maximal ample subbundle therefore we must have $\mathcal{A}_Y \subset \mathcal{A}_X, \ \mathcal{U}_Y' \subset \mathcal{U}_X'$ and similar for the last term. We also have the exact sequence 

\begin{equation}
h^*\left (f_*\omega_{Y/S}\right ) \to f_*\omega_{X/S} \to f_*\omega_{X/Y} \to 0
\end{equation}

\noindent
which states that the third identity is a quotient of the first one. Again this criterion applies componentwise. It follows that

\begin{equation}
\begin{aligned}
{rank}(\mathcal{A}_X)=\mathcal{A}_Y+\mathcal{A}_X^{var} &\leq {rank}(\mathcal{A}_Y) + {rank}(\mathcal{A}_{X/Y})\\
{rank}(\mathcal{U}_X) =\mathcal{U}_Y+\mathcal{U}_X^{var} &\leq {rank}(\mathcal{U}_Y) + {rank}(\mathcal{U}_{X/Y})\\
q_f =q_g+q_f^{var} &\leq q_g +q_h
\end{aligned}
\end{equation}

\noindent
That is $A_X^{var} \subset A_{X/Y}, \ \mathcal{U}_X^{var} \subset \mathcal{U}_{X/Y}$.

\begin{remark}
The proof of decomposition (79) is exactly the same as the lemma.
\end{remark}

\begin{remark}
The proof of the Lemma shows that $g_*\left (\theta^{2}\mathcal{H}_{\text{fix}}\right )=0$. 
\end{remark}

\begin{proposition}
We have 
\begin{itemize}
\item $H_e^{1,0} =H_e^{1,0}(1) \bigoplus H_e^{1,0}(2)_{\text{fix}}$.
\item $\mathcal{U}_f=\mathcal{U}_g \bigoplus \mathcal{U}_h^{\text{fix}}$.
\item $h^{1,0}=h(1)^{1,0}+h_{\text{fix}}^{1,0}(2)$. 
\item $h_0^{1,0}=h_0^{1,0}(1)+h_{0,\text{fix}}^{1,0}(2)$.
\item $\delta_0 = \delta_0^{1}+\left (\delta_0^{2}\right )_{\text{fix}}$.
\item $\nu(\theta)= \nu(\theta^{1})+\nu(g_*\theta^{2}\mid_{\text{fix}})$.
\item $\delta^{1,0}=\delta^{1}+\delta^{2}_{\text{fix}}=\delta^{1}+\delta^{2}-\delta^{2}_{\text{var}}$.
\item $u_f=u_g+u_h^{\text{fix}}=u_g+u_h-u_h^{\text{var}}$.
\end{itemize}
\end{proposition}

\begin{proof} By Proposition (5.1) and Theorem (5.2) we have the following commutative diagram with short exact rows

\begin{equation}
\begin{CD}
g_*\Omega_{Y/S}^1 (\log F)@>>>  f_*\Omega_{X/S}^1(\log G) @>>> f_*\Omega_{X/Y}^1(\log D)_{\text{fix}}   \\
@VV{\theta^{1}}V    @VV{\theta}V  @VVg_*{\theta^{2}}V \\
R^1g_*\mathcal{O}_Y \otimes \Omega_S^1(\log E) @>>> R^1f_*\mathcal{O}_X \otimes \Omega_S^1 (\log E) @>>> g_* (R^1h_*\mathcal{O}_X) \otimes \Omega_Y^1(\log E)_{\text{fix}} \\  
\end{CD}
\end{equation}

\noindent
Which induces the following split short exact sequence on the graded part of the local systems
 
\begin{equation}
\begin{CD}
H_e^{1,0}(1) @>>>  H_e^{1,0} @>>> H_e^{1,0}(2)_{\text{fix}} \\
  @VV{\theta^{1}}V    @VV{\theta}V  @VV{\theta^{2}}V \\
H_e^{0,1}(1) \otimes \Omega_S^1(\log E) @>>> H_e^{0,1} \otimes \Omega_S^1 (\log E)@>>> H_e^{0,1}(2)_{\text{fix}}^{\vee} \otimes \Omega_S^1 (\log E)   
\end{CD}
\end{equation}



\vspace{0.3cm}

\noindent
The diagram (84) is split exact for short exact sequences in rows. This proves that the kernels and cokernels also split, from which the identities in the lemma follows. 
\end{proof}

\begin{proposition} [see \cite{GG2} page 287] There is an induced map by the Gauss-Manin connections in the triangle 

\begin{equation}
\dfrac{g_*\left (R^1h_*\Omega_{X/S,var}^0\right )}{\theta^{2}_{\text{var}} \left ( g_* (R^0h_*\Omega_{X/S,var}^1) \right )} \ \stackrel{\overline{\theta_{\log}^{2,0}}}{\longrightarrow} \ \dfrac{R^2f_*\Omega_{X/S}^0 \otimes (\Omega_S^1)^{\otimes 2}}{\theta^{1,1}\left ( R^1f_*\Omega_{X/S}^1 \otimes \Omega_S^1 \right )}
\end{equation}

\noindent
which is injective.
\end{proposition}

\begin{proof}
The proof is a modification of the argument in [\cite{GG2} page 287, for product family]. However because we also need the details with other relations, we repeat the part loc. cit.. We will use the local variables $x=(v,u,s) \longmapsto y=(u,s) \mapsto s$ for the maps in (61). If we take a small enough neighborhood of $s \in S$ such that the two fibrations $X$ and $Y$ trivialize over the open set, then locally we have 

\begin{equation}
X_s=W_s \times F_s, \qquad F_s \cong \text{disc}
\end{equation} 

\noindent
as topological spaces, where $W_s$ are fibers of $X \to Y$, (here is different from \cite{GG2}, where there the fibered surfaces $X_s$ are self product of $Y_s$). In this way 

\begin{equation} 
U_s=X_s \diagdown W_s \qquad \cong \text{ tubular neighborhood of} \ W_s \ \text{in}\ X_s
\end{equation}

\noindent
We have a standard exact sequence

\begin{equation}
0 \to H^0(W_s) \to H^2(X_s) \to H^2(U_s) \stackrel{\text{Res}}{\rightarrow} H^1(W_s) \to H^3(X_s) \to 0
\end{equation}

\noindent
Then $\text{Res} H^2(U_s) =H_{\text{var}}^1(X_s)$. Moreover
 
\begin{equation}
\text{Res} F^iH^2(U_s) =F^{i-1}H_{\text{var}}^1(X_s)
\end{equation}

\noindent
Thus 

\begin{align}
\begin{aligned}
\frac{F^1H^2(U_s)}{F^2H^2(U_s)}&= H^1(\Omega_{X_s}^1(\log W_s) \\
\frac{F^0H^2(U_s)}{F^1H^2(U_s)}&=H^2(\mathcal{O}_{X_s})
\end{aligned}
\end{align}

\noindent
When $s$ varies the whole sequence (85) varies to give VMHS's. This shows [cf. \cite{GG2} the existence of a GM-connection 

\begin{equation}
\nabla_{\log}:R^1f_*\Omega_{X/S}^1(\log W_s) \otimes \Omega_S^1 \longrightarrow  R^2f_*\Omega_{X/S}^0(\log W_s) \otimes (\Omega_S^1)^{\otimes 2}
\end{equation}

\noindent 
We can illustrate all the GM-maps in the triangle by the following commutative diagram with short exact columns,

\begin{equation}
\begin{CD}
R^0f_*\Omega_{X/S}^2 @>{\nabla}>> R^1f_*\Omega_{X/S}^1 \otimes \Omega_S^1 @>{\nabla}>> R^2f_*\Omega_{X/S}^0 \otimes (\Omega_S^1)^{\otimes 2}\\
@VVV    @VVV  @VVV \\
R^0f_*\Omega_{X/S}^2(\log W_s) @>{\nabla}_{\log}>>  R^1f_*\Omega_{X/S}^1(\log W_s) \otimes \Omega_S^1 @>{\nabla}_{\log}>> R^2f_*\Omega_{X/S}^0(\log W_s) \otimes (\Omega_S^1)^{\otimes 2}  \\
@V{\text{Res}_{h=y}}VV    @V{\text{Res}_{h=y}}VV  @. \\
g_*R^0h_*\Omega_{X/S,var}^1 @>{\nabla}>>  g_*R^1h_*\Omega_{X/S,var}^0  
@.
\end{CD}
\end{equation}

\noindent
As it is explained in \cite{GG2} from the data of the above diagram and $\nabla_{\log}^{ 2}=0$, one can deduce the existence of the map
 
\begin{equation}
\overline{\nabla_{\log}}:\dfrac{g_*\left ( R^1h_*\Omega_{X/S,var}^0\right )}{\nabla_X \left (g_*R^0h_*\Omega_{X/S,var}^1\right)} \longrightarrow \dfrac{R^2f_*\Omega_{X/S}^0 \otimes (\Omega_S^1)^{\otimes 2}}{\nabla_X \left ( R^1f_*\Omega_{X/S}^1 \otimes \Omega_S^1 \right )}
\end{equation}

\noindent
From which we obtain the map (85).
\end{proof}

\begin{remark} $\nu(g_*\theta^{2}\mid_{\text{var}}) \leq \nu(\theta)$ follows from the (injectivity) of the map in (85). 
\end{remark}

\begin{theorem}
we have the following inequality on the degrees of the Hodge bundles in a family of surfaces fibred by curves
\begin{equation} 
\delta^{2,0}+2(\delta^{1,0}-(\delta^{1}+\delta^{2}) \geq (h^{1,1}-h^{2,0})(2g-2)^2-h^{1,0}(2)(2g-2)
\end{equation}
\end{theorem}
\begin{proof}
The map produced in (85) can be written as a short exact sequence

\begin{equation}
\begin{aligned}
0 \to  g_* (R^0h_*\Omega_{X/S,var}^1 ) \stackrel{{\theta}^{2}_{\text{var}}}{\longrightarrow} &g_*\left (R^1h_*\Omega_{X/S,var}^0 \right )\otimes \Omega_S^1 \ \stackrel{\overline{\theta_{\log}^{2,0}}}{\longrightarrow} \\
& \dfrac{R^2f_*\Omega_{X/S}^0 \otimes (\Omega_S^1)^{\otimes 2}}{\theta^{1,1} \left (R^1f_*\Omega_{X/S}^1 \otimes \Omega_S^1 \right )} \to Coker(\overline{\theta_{\log}^{2,0}}) \to 0
\end{aligned}
\end{equation}

\noindent
Using the additivity of the degree function on the category of coherent sheaves, we get

\begin{multline}
\deg(g_*R^0h_*\Omega_{X/S,var}^1) - \deg(g_*R^1h_*\Omega_{X/S,var}^0\otimes \Omega_S^1)+\\
 \deg(R^2f_*\Omega_{X/S}^0 \otimes (\Omega_S^1)^{\otimes 2})-
\deg(\theta^{1,1} R^1f_*\Omega_{X/S}^1 \otimes \Omega_S^1)-\deg \left ( Coker(\overline{\theta_{\log}^{2,0}}) \right )=0
\end{multline}

\noindent
Then calculating the degrees using the degree formula of the product gives the result (we are using the argument \cite{GGK} page 505-506 on calculation of degree in a general sequence). 
\begin{equation}
2\delta^{2}_{\text{var}}-(h^{1,0}(2))_{\text{var}}(2g-2)=\delta^{2,0}+(h^{2,0}-h^{1,1})(2g-2)^2+\sum_s \nu_s(\overline{\theta_{\log}^{2,0}})
\end{equation}
By substituting $\delta^{2}_{\text{var}}$ from proposition 5.6 we get
\begin{equation}
\delta^{2,0}+2(\delta^{1,0}-(\delta^{1}+\delta^{2})=(h^{1,1}-h^{2,0})(2g-2)^2-h^{1,0}(2)_{\text{var}}(2g-2)+\sum_s \nu_s(\overline{\theta_{\log}^{2,0}})
\end{equation}
from which the inequality of the Theorem follows.
\end{proof}

\begin{example} \cite{DM, GGK} We consider family of elliptic curves 

\begin{equation}
y^2=4x^3-g_2(t)x-g_3(t), \qquad t \in \mathbb{P}^1
\end{equation}

\noindent
where $g_2$ and $g_3$ are polynomials of degrees at most $4$ and $6$ respectively. 
Set $\Delta=g_2^3-27g_3^2, \ \ J=g_2^3/\Delta$ where $g_2, \ g_3, \ \Delta, \ J$ are Weierestrass coefficients, discriminant and $J$-function. The Picard-Fuchs equation for the family is given by 

\begin{equation}
\frac{d}{dt}\begin{bmatrix} \omega \\ \eta \end{bmatrix} =\begin{pmatrix} 
\frac{-1}{12}\frac{d \log \delta}{dt} & \frac{3\delta}{2 \Delta}\\
\frac{-g_2\delta}{8 \Delta} & \frac{1}{12}\frac{d \log \delta}{dt}
\end{pmatrix}\begin{bmatrix} \omega \\ \eta \end{bmatrix}, \qquad \delta=3g_3g_2'-2g_2g_3'
\end{equation}

\noindent
where $\omega=\int_{\gamma}\frac{dx}{y}, \  \eta =\int_{\gamma}\frac{xdx}{y}$ and $\gamma$ being a $1$-cycle, \cite{DM}. The equation (99) defines fibration 

\begin{equation}
\pi:X \to \mathbb{P}^1, \qquad X_t \mapsto t
\end{equation}

\noindent 
where the the local system $\mathcal{H}=R^1\pi_*\mathbb{C}$ has a two step Hodge filtration $F^0=H^1(X_t, \mathbb{C}) \supset F^1$. In this case the monodromy near $s_i$ can be written in terms of a canonical basis $\alpha, \ \beta $ of $H_1(X_t, \mathbb{C})$ as

\begin{equation}
\begin{aligned}
T(\alpha)&=\alpha\\
T(\beta)&=\beta+n_i\alpha
\end{aligned}
\end{equation} 
 
\noindent
A multi-valued section of $\mathcal{H}_e$ is

\begin{equation}
\phi(s)=\alpha^*+\left (n_i\frac{\log s}{2\pi \sqrt{-1}}+\psi(s) \right )\beta^*
\end{equation}

\noindent
where the action of the Gauss-Manin connection can be written as

\begin{equation}
\nabla \phi(s)=\left (n_i\frac{d s}{2\pi \sqrt{-1}s}+\psi'(s)ds \right )\beta^*
\end{equation}

\noindent
therefore 

\begin{equation}
N_i =Res_{s_i}\nabla=\begin{pmatrix} 0 & n_i\\ 0 &0 \end{pmatrix}
\end{equation}

\noindent 
An easy calculation [cf. \cite{GGK} loc cit.] gives 

\begin{equation}
\delta=\frac{1}{12}\sum_i n_i=0-1+\frac{1}{2}(N-\sum_{s_i}\nu_s(\theta))=\deg(J_t)
\end{equation}

\noindent 
If we calculate $N$ in the above formula we get

\begin{equation}
N=2 \delta +1 +\sum_i\nu_{s_i}(\theta) \geq 3
\end{equation} 

\noindent 
It follows that a non-trivial elliptic fibration over $\mathbb{P}^1$ has at least three singular fibers. Now we go to a family of $K3$ surfaces defined by

\begin{equation}
y^2=4x^3-G_2(t,s)x-G_3(t,s), \qquad \pi^2: X_{(t,s)} \to (t,s)
\end{equation}

\noindent
where $G_2, \ G_3$ are polynomials of degree at most $8, \ 12$ in the affine coordinate $s$ and such that they are also polynomials in $t$. The periods are calculated via the integral $\int_{\gamma} ds \wedge \frac{dx}{y}$ and the local systems $\mathcal{H}^2=R^2\pi_*^2\mathbb{C}$ has a weight $2$ Hodge filtration $F^0=H^2(X_{(t,s)},\mathbb{C}) \supset F^1 \supset F^2$. In this case still one has 

\begin{equation}
\deg(f_*(w_{X/S}))=\deg J_{t,s}
\end{equation}

\noindent
where $J$ is the $J$-function of the fibers, cf. \cite{GGK}. We have $\mathcal{H}_{var}=0, \ \mathcal{H}_{fix}=\mathcal{H}$. The formulas in proposition 5.6 get simplified. In both of the fibrations the matrix of $\bar{\theta}$ on the graded piece of middle cohomology is of the form (43). We easily note that

\begin{equation}
\delta^{1,0} \ne 0 \Rightarrow \delta^{2,0} \ne 0 \Rightarrow \text{Singular fibers exist}
\end{equation}

\noindent
By what was said, any family of elliptic curves parametrized by a complete curve must have at least 3 singular fibers. An illustration for the second family is 

\begin{equation}
\begin{CD}
E_t @>>> \mathcal{E} @>>>\mathbb{H} \\
@VVV  @VVV  @VV{j(\tau)}V\\
t @>>> \mathbb{P}^1  @>>{j(t)}>\mathbb{P}_{j-line}^1\\
\end{CD} 
\end{equation}

\noindent
One can apply the Riemann-Hurwitz ramification formula to obtain the interpretation of $\delta$ in terms of remification numbers of the fibration[see \cite{GGK} for details]. One may proceed inductively to 3-dimensional fibrations over surfaces, etc.
\end{example}

\begin{remark}
A triangle fibration can also be studied in higher dimensional fibrations, when suitable configuration is settled.
\end{remark}

\begin{remark} F. Catanese \cite{Ca, LPo, PPo} generalizes a theorem of Castelnuouvo-de Franchis for surfaces, so that if $\dim X=n$ and there are one forms 

\begin{equation}
\omega_i \in H^0(X,\Omega^1), \ i=1,2,...,k \qquad \text{such that} \ \omega_1 \wedge \omega_2 \wedge ... \wedge \omega_k \ne 0,
\end{equation} 

\noindent 
Then $X$ is fibered over a $k$-dimensional variety $Y$, The theorem can be used to extract certain inequalities involving Hodge numbers of fibrations and regularity, [see also \cite{Lo, GLa, CPi}]. 

\end{remark} 

\begin{theorem}
In a commutative triangle fibration of surfaces fibred by curves we have 

\begin{equation}
\delta_{X/S}^{2,0}\geq \frac{1}{2}(h_{X/Y}^{1,0}-h_{0,X/Y}^{1,0})(2g(S)-2)+ h_{X/Y}^{1,0}
\end{equation}

\noindent
Moreover we have $\delta_X^{2,0}\ \geq \ \delta_{X/Y}^{1,0}$ where the sub-indices denote the corresponding fibration.
\end{theorem}
\begin{proof}
In a fibration of the form (8) for a generic fiber we shall have $X_s \mapsto Y_s \mapsto s \in S$, where $Y_s$ is a curve and $X_s$ is a surface. The variety $X_s$ is also fibred over $Y_s$. Again a generic fiber over a point $p_s \in Y_s$ is a curve namely $W_{p_s}$. Therefore apart from a finite number of points of degeneracies on $S$ 

\begin{equation}
X_s=U_s \times W_{p_s}, \qquad s \in T \subset S.
\end{equation}

\noindent 
where $U_s$ is open in $Y_s$. In both of the directions of $s$ and $Y_s$ there are finitely many singular fibers. By (114)  the local system $\mathcal{H}$ obtained from the middle cohomology of the fibration $X \to S$ is given as 

\begin{equation}
\mathcal{H}=\ R^2f_*\mathbb{C}=\ \mathcal{V} \ \otimes \ \mathcal{L}, \qquad \mathcal{V}=g_* \left ( R^1h_*\mathbb{C} \right )
\end{equation}

\noindent 
where $\mathcal{V}$ and $\mathcal{L}$ are $wt=1$ HS. It follows that 

\begin{equation}
\mathcal{H}^{2,0}=\mathcal{V}^{1,0} \otimes \ \mathcal{L}^{1,0}
\end{equation}

\noindent 
Now we calculate the degrees of the associated Hodge bundles (denoted by the same symbols)

\begin{equation}
\begin{aligned}
\delta^{2,0}(\mathcal{H})&=\delta^{1,0}(\mathcal{V}) rank(\mathcal{L}^{1,0})+h^{1,0}(\mathcal{V}) \deg(\mathcal{L}^{1,0})\\
&=\left ( \frac{1}{2}(h_{X/Y}^{1,0}-h_{0,X/Y}^{1,0})(2g(S)-2)+\delta_{0,X/Y}-\frac{1}{2}\sum_{s \in S} \nu_s(\overline{\theta_{X/Y}})
 \right )\\ 
& \qquad \qquad \qquad \qquad \qquad  \times rk(\mathcal{L}^{1,0})
  + h^{1,0}(\mathcal{V})\deg(\mathcal{L}^{1,0}) 
\end{aligned}
\end{equation}

\noindent
from which (113) follows. Because $\deg(\mathcal{L}) \geq 0$, it follows that $\delta^{2,0} \geq \delta^{1,0}(\mathcal{V})$. 
\end{proof}

\begin{theorem}
If the local system $\mathcal{V}=g_* \left (R^1h_*\mathbb{C} \right )$ is irreducible with regular singularities, then 

\begin{equation}
h^p(H^1(S, \mathcal{V}))=-\delta^p(V)-\left ( h^p(V)-h^{p-1}(V) \right )(2g(S)-2)+\delta^{p-1}+\sum_s\nu_s(\bar{\theta})(V)
\end{equation}

\noindent
where $V=\mathcal{V} \otimes \mathcal{O}_S$ and $h^i(V)=\dim Gr_F^i(V)$. 
\end{theorem}

\begin{proof}
Set $\mathcal{V} =g_* \left (R^1h_*\mathbb{C} \right )$ and $V=\mathcal{V} \otimes \mathcal{O}_S$. Consider the resolution  

\begin{equation}
{H}^{\bullet}: 0 \to \mathcal{V} \longrightarrow \mathcal{V} \otimes \mathcal{O}_S \stackrel{\nabla}{\longrightarrow} \mathcal{V} \otimes \Omega_{S}^1
\end{equation}

\noindent 
Then by our assumption its only non-zero cohomology $H^1(S,\mathcal{V})$ is a HS of $wt=2$, 

\begin{equation}
F^p(\mathcal{V} \otimes \Omega_S^{\bullet})=F^{p-\bullet}\mathcal{V} \otimes \Omega_S^{\bullet}
\end{equation}

\noindent 
We know that the spectral sequence 

\begin{equation} 
\mathbb{H}^{r+s}(S,Gr_F^r(\Omega_S^{\bullet} \otimes\mathcal{V})) \Rightarrow \mathbb{H}^{r+s}(S, \mathcal{V})
\end{equation} 

\noindent 
degenerates at $E_1$. Computing the Euler characteristics on $Gr_F^p$ of (119), gives

\begin{equation}
\begin{aligned}
h^p\left ( H^1(S, \mathcal{V}) \right )&=\chi(Gr_F^p\mathbb{H}^{\bullet}(S,V))=\chi(\mathbb{H}^{\bullet}(S,Gr_F^pV)), \qquad \text{degeneration at}\ E_1 \\
&=-\chi(S, Gr_F^pV)+\chi(S,\Omega_S^1 \otimes Gr_F^{p-1}V)
\end{aligned}
\end{equation}

\noindent 
Applying the Riemann-Roch theorem we get 

\begin{equation}
h^p(H^1(S, \mathcal{V}))=-\delta^p(V)-h^p(V)(2g(S)-2)+\delta^{p-1}+h^{p-1}(V)(2g(S)-2)+\sum_s\nu_s(\bar{\theta})(V)
\end{equation}

\noindent
where we have used $\deg(gr_F^{p-1}V)=\delta^{p-1}(V)+\sum_s\nu_s(\bar{\theta})(V)$.
 
\end{proof}

\begin{remark}
A version of Theorem 5.14 over $\mathbb{P}^1$ for tensor product of a (single) VHS with a line bundle is proved in \cite{SD}. It is also possible to formulate similar formulas for the fibrations $X/S, \ Y/S$ separately. 
\end{remark}

\vspace{0.3cm}

\textbf{Declaration of interest}: There is no declaration of interests by authors.


\begin{thebibliography}{99}

\bibitem{AGV}  V. Arnold, S. Gusein Zade, A. Varchenko;  Singularities of differentiable maps, Vol 2, Monodromy and asymptotics of integrals, 1984

\bibitem{B}  E. Brieskorn Die Monodromie der isolierten Singularitaten von Hyperflachen. Man. Math. 2 (1970) 103-161

\bibitem{CK} F. Catanese, Y. Kawamata, Fujita decomposition over higher dimensional base, European Journal of Mathematics, pp 1–9, 2018 

\bibitem{CaDe3} Catanese, F.; Dettweiler, M. The direct image of the relative dualizing sheaf needs
not be semiample. C. R. Math. Acad. Sci. Paris 352 (2014), no. 3, 241–244.

\bibitem{CaDe2} Catanese, F.; Dettweiler, M. Answer to a question by Fujita on Variation of Hodge
Structures. Advanced Studies in Pure Mathematics 74-04 ”Higher Dimensional Algebraic Geometry”, dedicated to Kawamata on his 60th birthday, 73–102 (2017).

\bibitem{CaDe} Catanese, F.; Dettweiler, M. Vector bundles on curves coming from variation of
Hodge structures. Internat. J. Math. 27 (2016), no. 7, 1640001, 25 pp.


\bibitem{Ca} F. Catanese, Moduli and classification of irregular Kaehler manifolds (and algebraic varieties) with Albanese general type fibrations, Invent. Math. 104 (1991), 389–407.

\bibitem{CPi} A. Causin and G. P. Pirola, Hermitian matrices and cohomology of Kahler varieties, Manuscripta Math. 121 (2006), 157–168.

\bibitem{Cl} Clemens J., C.H.: Picard-Lefschetz theorem for families of nonsingular algebraic varieties acquiring ordinary singularities. Trans. A.M.S. 136, 93-108 (1969)

\bibitem{D1}  P. Deligne, Theorie de Hodge. II. Publications Mathematiques de l'IHES (1971) 
Volume: 40, page 5-57

\bibitem{DM} C. Doran and A. Malmendier, Calabi-Yau Manifolds Realizing Symplectically Rigid Monodromy Tuples. Advances in Theoretical and Mathematical Physics, Volume 23, Issue 5.

\bibitem{Fuj2} Fujita, T. On Kahler fiber spaces over curves. J. Math. Soc. Japan 30 (1978), no. 4, 779–794.

\bibitem{Fuj1} Fujita, T. The sheaf of relative canonical forms of a Kahler fiber space over a curve. Proc.
Japan Acad. Ser. A Math. Sci. 54 (1978), no. 7, 183–184.

\bibitem{GLa} M. Green and R. Lazarsfeld, Deformation theory, generic vanishing theorems, and some conjectures of Enriques, Catanese and Beauville, Invent.
Math 90 (1987), 389–407.

\bibitem{G1}  P. A. Griffiths, Periods of integrals on algebraic manifolds. II. Local study of the period mapping. Amer. J. Math., 90: 805-865, 1968.
 
\bibitem{G3}  P. Griffiths , J. Harris. Principles of algebraic geometry. Wiley Classics Library. John Wiley-Sons Inc., New York, 1994.

\bibitem{G4}  P. Griffiths;  Hodge theory and representation theory,  10 Lec. at TCU, 2012

\bibitem{G} V. Gonzalez, Bounding Hodge numbers of irregular varieties, Presentation Joint with G. P.Pirola, BarcelonaTech, Spain

\bibitem{GGK} M. Green, P. Griffiths, M. Kerr, Some enumerative  global properties of variation of Hodge structures, Moscow Math Journal, Vol 8, N. 3, 2009, 469-530


\bibitem{GG2} P. Griffiths, M. Green, An interesting 0-cycle, Duke Mathematical Journal, Vol 119, No 2, 261-313, 2003

\bibitem{Har} Hartshorne, R.: Residues and duality. Lecture Notes in Math. 20. Berlin-Heidelberg-New York:
Springer 1966

\bibitem{Hod} Hodge, W. V. D.: Theory and applications of harmonic integrals. Cambridge: Cambridge University Press 1963

\bibitem{JS1}  J. Steenbrink: Limits of Hodge structures. Invent. math. 31 (1976) 229-257

\bibitem{JS2} J. Steenbrink: Mixed Hodge structure on the vanishing cohomology. In: P. Holm (ed.): Real and complex Singularities. Oslo 1976. Sijthoff-Noordhoff, 
1977, pp. 525-563

\bibitem{JS3} J. Steenbrink: Mixed Hodge structures associated with isolated singularities. Proc. Symp. Pure Math. 40, Part 2 (1983) 513-536

\bibitem{Kaw1} Kawamata, Y. Variation of mixed Hodge structures and the positivity for algebraic fiber
spaces. Algebraic Geometry in East Asia  Taipei 2011. Adv. St. Pure Math. 65(2015), 27–58.

\bibitem{LPo} R. Lazarsfeld and M. Popa, Derivative complex, BGG correspondence, and numerical inequalities of compact Kahler manifolds , Invent. Math. 182, no.3, (2010), 605–633.

\bibitem{Lo} L. Lombardi, Inequalities for the Hodge numbers of irregular compact Kahler manifolds, in preparation.

\bibitem{PPo} G. Pareschi and M. Popa, Strong Generic Vanishing and a higher dimensional Castelnuovo-de Franchis inequality, Duke Math. J. 150 no.2 (2009), 269–285

\bibitem{KUL}   V. Kulikov, Mixed Hodge structure and singularities, Cambridge University Press, 1998

\bibitem{PS}  Peters C. , Steenbrink J. , Mixed Hodge structures, A series of modern surveys in mathematics. Springer verlag publications, Vol 52, 2007

\bibitem{PT} G. P. Pirola, S. Torelli, Massey Products and Fujita decompositions on fibrations of curves, Collectanea Mathematica, pp 1–23, 2019

\bibitem{Sch1}	W. Schmid, "Variation of Hodge structure: the singularities of the period mapping" Invent. Math. , 22 (1973) pp. 211–319 

\bibitem{St1}	J. Steenbrink, S. Zucker, "Variation of mixed Hodge structure, I" Invent. Math. , 80 (1985) pp. 489–542 

\bibitem{SC2}  J. Scherk, J. Steenbrink: On the mixed Hodge structure on the cohomology of the Milnor fibre. Math. Ann. 271 (1985) 641-665

\bibitem{SCHU}     M. Schultz,  Algorithmic Gauss-Manin connection, Algorithms to compute Hodge-theoretic invariants of isolated hypersurface singularities, Ph.D dissertation, Universitat Kaiserslautern. 2002

\bibitem{SD} M. Dettweiler and C. Sabbah, Hodge theory of the middle convolution, Publ. RIMS, Kyoto Univ. 49 (2013), 761–800 

\bibitem{V} A. Varchenko,  On the local residue and the intersection form on the vanishing cohomology, Math. USSR Izv. Akad. Nauk SSSR Ser. Mat., 1985, Volume 49, Issue 1, Pages 32-54,  1986

\bibitem{Vi} Viehweg, Eckart, Weak positivity and the additivity of the Kodaira dimension for certain fibre
spaces. Algebraic varieties and analytic varieties (Tokyo, 1981), 329–353, Adv. Stud. Pure Math. 1, North-Holland, Amsterdam, (1983).

\bibitem{VZ} E. Viehweg, K. Zuo, Families over curves with a strictly maximal Higgs field, Asian J. Math. 7 (2003) 575-598.

\bibitem{VZ1} Viehweg, E., Zuo, K.: On the isotriviality of families of projective manifolds over
curves. J. Alg. Geom. 10 (2001) 781-799.
\bibitem{VZ2} Viehweg, E., Zuo, K.: Families over curves with a strictly maximal Higgs field.
Asian J. of Math. 7 (2003) 575-598.
\bibitem{VZ3} Viehweg, E., Zuo, K.: A characterization of certain Shimura curves in the moduli
stack of abelian varieties. J. Diff. Geom. 66 (2004) 233-287
\bibitem{VZ4} Viehweg E., Zuo K., Numerical bounds for semistable families of curves or of certain higher dimensional manifolds, Journal of Algebraic Geometry 2005, 15(4)

\end{thebibliography}
\end{document}